\documentclass{amsart}
\usepackage{amssymb}
\usepackage{graphicx}

\newtheorem{thm}{Theorem}
\newtheorem{cor}{Corollary}
\newtheorem{lem}{Lemma}

\newtheorem{rem}{Remark}
\newtheorem{conj}{Conjecture}

\newcommand{\A}{{\mathcal A}}

\newcommand{\es}{{\mathcal S}}

\newcommand{\D}{{\mathbb D}}

\newcommand{\real}{{\operatorname{Re}\,}}

\def\be{\begin{equation}}
\def\ee{\end{equation}}

\newcommand{\bee}{\begin{enumerate}}
\newcommand{\eee}{\end{enumerate}}

\newcommand{\blem}{\begin{lem}}
\newcommand{\elem}{\end{lem}}
\newcommand{\bthm}{\begin{thm}}
\newcommand{\ethm}{\end{thm}}
\newcommand{\bcor}{\begin{cor}}
\newcommand{\ecor}{\end{cor}}
\newcommand{\beg}{\begin{example}}
\newcommand{\eeg}{\end{example}}
\newcommand{\begs}{\begin{examples}}
\newcommand{\eegs}{\end{examples}}
\newcommand{\bdefe}{\begin{defin}}
\newcommand{\edefe}{\end{defin}}
\newcommand{\bprob}{\begin{prob}}
\newcommand{\eprob}{\end{prob}}
\newcommand{\bei}{\begin{itemize}}
\newcommand{\eei}{\end{itemize}}

\newcommand{\bcon}{\begin{conj}}
\newcommand{\econ}{\end{conj}}
\newcommand{\bcons}{\begin{conjs}}
\newcommand{\econs}{\end{conjs}}
\newcommand{\bprop}{\begin{propo}}
\newcommand{\eprop}{\end{propo}}
\newcommand{\br}{\begin{rem}}
\newcommand{\er}{\end{rem}}
\newcommand{\brs}{\begin{rems}}
\newcommand{\ers}{\end{rems}}
\newcommand{\bo}{\begin{obser}}
\newcommand{\eo}{\end{obser}}
\newcommand{\bos}{\begin{obsers}}
\newcommand{\eos}{\end{obsers}}
\newcommand{\bpf}{\begin{pf}}
\newcommand{\epf}{\end{pf}}
\newcommand{\ba}{\begin{array}}
\newcommand{\ea}{\end{array}}
\newcommand{\beq}{\begin{eqnarray}}
\newcommand{\beqq}{\begin{eqnarray*}}
\newcommand{\eeq}{\end{eqnarray}}
\newcommand{\eeqq}{\end{eqnarray*}}

\begin{document}
\bibliographystyle{amsplain}

\title[Some properties of logarithmic coefficients of inverse functions]{On some properties of logarithmic coefficients of inverse of univalent functions}

\author[M. Obradovi\'{c}]{Milutin Obradovi\'{c}}
\address{Department of Mathematics,
Faculty of Civil Engineering, University of Belgrade,
Bulevar Kralja Aleksandra 73, 11000, Belgrade, Serbia}
\email{obrad@grf.bg.ac.rs}

\author[N. Tuneski]{Nikola Tuneski}
\address{Department of Mathematics and Informatics, Faculty of Mechanical Engineering, Ss. Cyril and Methodius
University in Skopje, Karpo\v{s} II b.b., 1000 Skopje, Republic of North Macedonia.}
\email{nikola.tuneski@mf.edu.mk}

\author[P. Zaprawa]{Pawe{\l} Zaprawa}
\address{Department of Mathematic, Faculty of Mechanical Engineering, Lublin University of Technology, Poland.}
\email{p.zaprawa@pollub.pl}



\begin{abstract}
In this paper we consider some properties of the initial logarithmic coefficients for inverse functions of functions univalent in the unit disc. The case of convex functions is treated separately. We give estimate, in some cases sharp, of the modulus of the initial coefficients, as well as the difference of the modulus of two consecutive coefficients.
\end{abstract}

\maketitle

\section{Introduction and definitions}

\medskip

Let $\mathcal{A}$ be the class of functions $f$ analytic in the open unit disk $\D=\{z:|z|<1\}$ of the form
\be\label{eq-1}
f(z)=z+a_2z^2+a_3z^3+\cdots,
\ee
and let $\mathcal{S}$ be the subclass of $\mathcal{A}$ consisting of functions that are univalent in $\D$.

\medskip

the biggest challenge intheory of univalent functions, the famous Bieberbach conjecture $|a_n|\le n$ for $n\ge2,$ was proven by de Branges in 1985 \cite{Bra85}, but still great many problems concerning the coefficients $a_n$ remain open and are studied over entire class $\es$, or, over its subclasses.

\medskip

One such problem is finding sharp estimates of the modulus of the logarithmic coefficient, $\gamma_n$, of an univalent function $f(z)=z+a_2z^2+a_3z^3+\cdots$ from $\es$,  defined by
  \be\label{log-co-1}
  F_f(z) := \log\frac{f(z)}{z}=2\sum_{n=1}^\infty \gamma_n z^n.
  \ee
From the relations \eqref{eq-1} and \eqref{log-co-1}, after  equating the coefficients for initial logarithmic coefficients we receive:
\begin{equation}\label{eq-3}
\begin{split}
 \gamma_{1}&=\frac{a_{2}}{2},\\
 \gamma_{2}&=\frac{1}{2}\left(a_3-\frac{1}{2}a_{2}^{2}\right), \\
 \gamma_{3}&=\frac{1}{2}\left(a_4-a_2a_3+\frac13a_2^3\right).
\end{split}
\end{equation}

\medskip

Relatively little exact information is known about the coefficients. The natural conjecture $|\gamma_n|\le1/n$, inspired by the Koebe function (whose logarithmic coefficients are $1/n$) is false even in order of magnitude (see Duren \cite[Section 8.1]{duren}).
For the class $\es$ the sharp estimates of single logarithmic coefficients S are known only for $\gamma_1$ and $\gamma_2$, namely,
\[|\gamma_1|\le1\quad\mbox{and}\quad |\gamma_2|\le \frac12+\frac1e=0.635\ldots,\]
and are unknown for $n\ge3$. For estimates of $|\gamma_{3}|$ and $|\gamma_{4}|$ we can find in \cite{106} and \cite{114}.

\medskip

Further, due to the famous Koebe's 1/4 theorem, a function $f\in\es$ given by \eqref{eq-1} has its inverse at least on the disk with radius 1/4 with an expansion
\begin{equation}\label{eq-4}
f^{-1}(w) = w+A_2w^2+A_3w^3+\cdots.
\end{equation}
Using the identity $f(f^{-1}(w))=w$, and after comparing the appropriate coefficient, we receive
\begin{equation}\label{eq-5}
\begin{array}{l}
A_{2}=-a_{2}, \\
A _{3}=-a_{3}+2a_{2}^{2} , \\
A_{4}= -a_{4}+5a_{2}a_{3}-5a_{2}^{3}.
\end{array}
\end{equation}

\medskip

In this paper we give estimates of the modulus for the initial logarithmic coefficients $\Gamma_{1},\Gamma_{2},\Gamma_{3}$ of the inverse function of normalized univalent  functions from the class $ \mathcal{S}$. Also, some problems concerning difference of the modulus of those  logarithmic coefficients  will be considered.

\medskip

For our investigation we  will use the Grunsky coefficients and their properties. We continue with the basic definitions and results  on those coefficients based on the book of N. A. Lebedev (\cite{Lebedev}).

\medskip

Let $f \in \mathcal{S}$ and let
\[
\log\frac{f(t)-f(z)}{t-z}=\sum_{p,q=0}^{\infty}\omega_{p,q}t^{p}z^{q},
\]
where $\omega_{p,q}$ are called Grunsky's coefficients with property $\omega_{p,q}=\omega_{q,p}$.
For those coefficients we have the next Grunsky's inequalitiy (\cite{duren,Lebedev}):
\be\label{eq-6}
\sum_{q=1}^{\infty}q \left|\sum_{p=1}^{\infty}\omega_{p,q}x_{p}\right|^{2}\leq \sum_{p=1}^{\infty}\frac{|x_{p}|^{2}}{p},
\ee
where $x_{p}$ are arbitrary complex numbers such that $0< \sum_{p=1}^{\infty}\frac{|x_{p}|^{2}}{p}< +\infty$. If $\overline{\lim}_{p\to\infty} \sqrt[p]{|x_p|}<1$, then in \eqref{eq-6} we have equality if, and only if, the area of $\widehat{\mathbb{C}} \setminus f^{-1}(\D)$ is zero, where $f^{-1}(z)= \frac{1}{f(z)}$.

\medskip

Further, it is well-known that if $f$ given by \eqref{eq-1}
belongs to $\mathcal{S}$, then also
\[
f_{2}(z)=\sqrt{f(z^{2})}=z +c_{3}z^3+c_{5}z^{5}+\cdots
\]
belongs to the class $\mathcal{S}$. Then for the function $f_{2}$ we have the appropriate Grunsky's
coefficients of the form $\omega_{2p-1,2q-1}$ and the inequality \eqref{eq-6} has the form:
\be\label{eq-8}
\sum_{q=1}^{\infty}(2q-1) \left|\sum_{p=1}^{\infty}\omega_{2p-1,2q-1}x_{2p-1}\right|^{2}\leq \sum_{p=1}^{\infty}\frac{|x_{2p-1}|^{2}}{2p-1}.
\ee

\medskip

Here, and further in the paper we omit the upper index (2) in  $\omega_{2p-1,2q-1}^{(2)}$ if compared with Lebedev's notation.

\medskip

From the inequality \eqref{eq-8}, when $x_{2p-1}=0$ and $p=3,4,\ldots$, we have
\begin{equation}\label{eq-9}
|\omega_{11} x_1 +\omega_{31} x_3 |^2 +3|\omega_{13} x_1 +\omega_{33} x_3 |^2 + 5|\omega_{15} x_1 +\omega_{35} x_3 |^2 \le |x_1|^2+\frac{|x_3|^2}{3}.
\end{equation}
 From \eqref{eq-9}, for $x_1=1$ and $x_3=0$, we obtain
\[
  |\omega_{11}|^2 + 3 |\omega_{13}|^2 + 5|\omega_{15}|^2 \leq1.
\]

Moreover, from \eqref{eq-9}, for $x_{1}=0$ and  $x_{3}=1 $, omitting the last component and applying $\omega_{31}=\omega_{13}$, we get
\begin{equation}\label{eq-10a}
|\omega_{13}|^{2}+3|\omega_{33}|^{2}\leq\frac13.
\end{equation}

As it has been shown in \cite[p. 57]{Lebedev}, if $f$ is given by \eqref{eq-1} then the coefficients $a_{2}$, $ a_{3}$, $ a_{4}$  are expressed by Grunsky's coefficients  $\omega_{2p-1,2q-1}$ of the function $f_{2}$ given by
\eqref{eq-6} in the following way:
\begin{equation}\label{eq-11}
\begin{split}
a_{2}&=2\omega _{11},\\
a_{3}&=2\omega_{13}+3\omega_{11}^{2}, \\
a_{4}&=2\omega_{33}+8\omega_{11}\omega_{13}+\frac{10}{3}\omega_{11}^{3},\\
0&=3\omega_{15}-3\omega_{11}\omega_{13}+\omega_{11}^{3}-3\omega_{33}.
\end{split}
\end{equation}

\medskip

A comprehensive overview of the application of Grunsky coefficients in the general class of univalent functions is given in \cite{128}.

\medskip

In what follows we need a basic inequality
\begin{equation}\label{eq-99}
|a_{3}-a_{2}^{2}|\leq1
\end{equation}
valid for all functions $f\in\mathcal{S}$. This and \eqref{eq-11} result in
\begin{equation}\label{eq-100}
|2\omega_{13}-\omega_{11}^{2}|\leq1.
\end{equation}

\medskip

\section{The initial  logarithmic coefficients of inverse of univalent functions}

\medskip

The first theorem will give estimates of the modulus $|\Gamma_{1}|$, $|\Gamma_{2}|$, and $|\Gamma_{3}|.$

\medskip

\begin{thm}\label{th1}
Let $f\in\mathcal{S}$ and let $\Gamma_{1},\Gamma_{2},\Gamma_{3}$ be initial logarithmic coefficients of
its inverse function $f^{-1}$ given by \eqref{eq-5}. Then the following sharp inequalities hold
\begin{itemize}
  \item[($i$)] $|\Gamma_{1}|\leq 1$;
  \item[($ii$)] $|\Gamma_{2}| \leq\frac{3}{2}$;
  \item [($iii$)]$|\Gamma_{3}|\leq\frac{10}{3}$.
\end{itemize}
\end{thm}

\begin{proof}
Let $f\in\mathcal{S}$ is given by \eqref{eq-1} and its inverse $f^{-1}$ is given by \eqref{eq-4}, then using the relations \eqref{eq-3} and \eqref{eq-5}, we have
\[ |\Gamma_{1}|=\frac{1}{2}|A_{2}|=\frac{1}{2}|-a_{2}|\leq 1\]
and
\[|\Gamma_{2}|=\frac{1}{2}|A_{3}-\frac{1}{2}A_{2}^{2}|=\frac{1}{2}|-a_{3}+\frac{3}{2}a_{2}^{2}|
\leq \frac{1}{2}(|a_{3}-a_{2}^{2}|+\frac{1}{2}|a_{2}|^{2})\leq\frac{3}{2},\]
where we used \eqref{eq-99}.

\medskip

This proved (i) and (ii).

\medskip

For proving (iii), similarly, we obtain
\[
\Gamma_{3}=\frac{1}{2}\left( A_4-A_2A_3+\frac{1}{3}A_2^3\right)=\frac{1}{2}\left(-a_{4}+4a_{2}a_{3}-\frac{10}{3}a_{2}^{3}\right),
\]
or if we use the Grunsky coefficients given by relations \eqref{eq-11}:
\[\Gamma_{3} = -\omega_{33}+4\omega_{11}\omega_{13}-3\omega_{11}^{3} =-\omega_{33}+2\omega_{11}(2\omega_{13}-\omega_{11}^2)-\omega_{11}^{3}.\]
From \eqref{eq-100} and \eqref{eq-10a}, we have
\[
\begin{split}
|\Gamma_{3}| &\leq |\omega_{33}|+2|\omega_{11}|+|\omega_{11}|^{3} \\
&\leq \frac13\sqrt{1-3|\omega_{13}|^2}+2|\omega_{11}|+|\omega_{11}|^{3} \\
&\leq \frac{10}{3}\ .
\end{split}
\]

\medskip

All three estimates are sharp as the Koebe function $k(z)=\frac{z}{(1-z)^{2}}$ shows.
\end{proof}

\medskip

\section{On differences of initial logarithmic coefficients for inverse functions}

\medskip

In this section, using similar method as in Theorem \ref{th1}, we will obtain lower and upper bound (in some cases sharp) of the differences of initial logarithmic coefficients for inverse functions: $|\Gamma_{2}|-|\Gamma_{1}|$ and $|\Gamma_{3}|-|\Gamma_{2}|$.

\medskip

\begin{thm}\label{th2}
Let $f\in\mathcal{S}$ be given by \eqref{eq-1}, and let $\Gamma_{1}$ and $\Gamma_{2}$
be the first two initial logarithmic coefficients of the function $f^{-1}$. Then
$$-\frac{\sqrt{2}}{2}\leq |\Gamma_{2}|-|\Gamma_{1}|\leq \frac{1}{2}.$$
The upper bound is sharp.
\end{thm}

\medskip

\begin{proof}
As in the proof of Theorem \ref{th1}, we have $ |\Gamma_{1}|=\frac{1}{2}|A_{2}|=\frac{1}{2}|-a_{2}|$ and
$$|\Gamma_{2}|=\frac{1}{2}\left|A_{3}-\frac{1}{2}A_{2}^{2} \right|=\frac{1}{2}\left|-a_{3}+\frac{3}{2}a_{2}^{2}\right|.$$
Now, since $\frac{|a_{2}|}{2}\leq 1$ for every $f\in \mathcal{S}$, 
\[
\begin{split}
|\Gamma_{2}|-|\Gamma_{1}|&\leq|\Gamma_{2}|-\frac{1}{2}|a_{2}||\Gamma_{1}|\\
&= \frac{1}{2}\left|-a_{3}+\frac{3}{2}a_{2}^{2}\right|-\frac{1}{4}|a_{2}|^{2}\\
&\leq \frac{1}{2}\left|-a_{3}+\frac{3}{2}a_{2}^{2}-\frac{1}{2}a_{2}^{2}\right|\\
&=\frac{1}{2}|-a_{3}+a_{2}^{2}|\leq \frac{1}{2},
\end{split}
\]
where we used \eqref{eq-99}. The inequality is sharp for the inverse of the Koebe function.

\medskip

The left hand side of the inequality of this theorem is equivalent with
$$\frac{1}{2}\left|-a_{3}+\frac{3}{2}a_{2}^{2}\right|-\frac{1}{2}|a_{2}|\geq  -\frac{\sqrt{2}}{2},$$
i.e., with
\begin{equation}\label{eq-13}
\left|-a_{3}+\frac{3}{2}a_{2}^{2}\right|\geq|a_{2}| -\sqrt{2}.
\end{equation}

\medskip

If $0\leq |a_{2}|<\sqrt{2}$, then \eqref{eq-13} is evident. 

\medskip

So, let $\sqrt{2}\leq |a_{2}|\leq2$.
Then, since
$$\left|-a_{3}+\frac{3}{2}a_{2}^{2}\right|\geq\frac{1}{2}|a_{2}|^{2}-|a_{3}-a_{2}^{2}|\geq\frac{1}{2}|a_{2}|^{2}-1,$$
in order to prove the relation  \eqref{eq-13}, it is enough to show that
$$ \frac{1}{2}|a_{2}|^{2}-1\geq |a_{2}| -\sqrt{2}.$$
The last inequality is equivalent to
$$\left[|a_{2}| -\sqrt{2}\right] \left[\left(|a_{2}| +\sqrt{2}\right)-2\right]\geq0,$$
which is true since by assumption, $|a_{2}| -\sqrt{2}\geq 0$  and also, $\left(|a_{2}| +\sqrt{2}\right)-2\geq 2\sqrt{2}-2>0$.
This completes the proof.
\end{proof}

\medskip

\begin{thm}\label{th3}
Let $f\in\mathcal{S}$ be given by \eqref{eq-1}, and let $\Gamma_{2}$ and $\Gamma_{3}$
be the  logarithmic coefficients of its inverse function $f^{-1}$. Then
$$ |\Gamma_{3}|-|\Gamma_{2}|\leq\frac{11}{6}.$$
The bound is sharp.
\end{thm}

\medskip

\begin{proof}
From the proof of parts (ii) and (iii) of Theorem \ref{th1} we have
$$\Gamma_{2}=\frac{1}{2}\left(A_{3}-\frac{1}{2}A_{2}^{2}\right) = \frac{1}{2}\left(-a_{3}+\frac{3}{2}a_{2}^{2}\right)$$
and
$$\Gamma_{3}=\frac{1}{2}\left(A_4-A_2A_3+\frac13A_2^3\right)=\frac{1}{2}\left(-a_{4}+4a_{2}a_{3}-\frac{10}{3}a_{2}^{3}\right).$$
Since $|a_{2}|/2\leq 1$,
\[
\begin{split}
|\Gamma_{3}|-|\Gamma_{2}|&\leq |\Gamma_{3}|-\frac{|a_{2}|}{2}|\Gamma_{2}|\leq \left|\Gamma_{3}-\frac{a_{2}}{2}\Gamma_{2}\right|\\
&=\frac{1}{2}\left|-a_{4}+\frac{7}{2}a_{2}a_{3}-\frac{31}{12}a_{2}^{3}\right|.
\end{split}
\]
Using the relations \eqref{eq-11}, after some computations, we get
$$|\Gamma_{3}|-|\Gamma_{2}|\leq \left|\omega_{33}-3\omega_{11}\omega_{13}+\frac{3}{2}\omega_{11}^{3}\right|\leq |\omega_{33}|+\frac32|\omega_{11}|\cdot|2\omega_{13}-\omega_{11}^{2}|.$$
Hence, applying \eqref{eq-10a} and \eqref{eq-100},
\[|\Gamma_{3}|-|\Gamma_{2}|\leq \frac13\sqrt{1-3|\omega_{13}|^2}+\frac32|\omega_{11}|\leq \frac{11}{6}\ .
\]
Note that for Koebe function we have $|\Gamma_{2}|=\frac32$ and $|\Gamma_{3}|=\frac{10}{3}$, so in this case
\[|\Gamma_{3}|-|\Gamma_{2}|=\frac{11}{6}\]
showing the sharpness of the estimate.
\end{proof}

\medskip

\section{Convex case of inverse functions}

\medskip

In this part we consider the previous investigations but for inverse functions of convex functions in the unit disc.
The class of convex functions is defined by
$$\mathcal{C}=\left\{ f\in\A: \real\left[1+  \frac{zf''(z)}{f'(z)}\right]>0,\,z\in\D\right\},$$
and it is subclass of the class $\es$.
We note that for convex functions we have $|a_{n}|\leq 1, \, n=2,3,4,\ldots.$
More details can be found in \cite{duren}. 

\medskip

For our investigations we need the following two lemmas.

\medskip

\begin{lem}\label{lem-1} \cite{127} For $f\in \mathcal{C}$ and $f(z)=z+a_2z^2+a_3z^3+\cdots $, we have
$$ |a_3-a_2^{2}|\leq \frac{1}{3}(1-|a_2|^{2}).$$
The inequality is sharp with extremal function
$$f_{\lambda}(z)=\int_{0}^{z}\left(\frac{1+t}{1-t}\right)^{\lambda}\frac{1}{1-t^{2}}dt=z+\lambda z^{2}
+\frac{1}{3}(2\lambda^{2}+1)z^{3}+\frac{1}{3}(\lambda^{3}+2\lambda)z^{4}+\cdots,$$
where $0\leq \lambda \leq1.$
\end{lem}

\medskip

\begin{lem}\label{lem-2} \cite{126} If $\omega(z)=c_{1}z+c_2z^2+c_3z^3+\cdots $ is analytic in $\D$, satisfies the condition
$|\omega(z)|<1$, $z\in\D $, and if
$$ \Psi(\omega)=|c_{3}+\mu c_{1}c_{2}+\nu c_{1}^{3}|,$$
then the following sharp estimate $\Psi(\omega)\leq \Phi(\mu, \nu)$ holds, where
\[\Phi(\mu,\nu)=
\begin{cases}
|\nu|, & (\mu,\nu)\in D_{6},\\
 \frac{2}{3}(|\mu|+1)\left(\frac{|\mu|+1}{3(|\mu|+1+\nu)}\right)^{\frac{1}{2}}, & (\mu,\nu)\in D_{8}\ ,
\end{cases}\]
with
$$D_{6}=\left\{(\mu, \nu): 2\leq|\mu|\leq4,\,\nu\geq\frac{1}{12}(\mu^{2}+8) \right\}$$
and
 $$D_{8}= \left\{(\mu, \nu): \frac{1}{2}\leq|\mu|\leq2,\,
-\frac{2}{3}(|\mu|+1)\leq\nu\leq \frac{4}{27}(|\mu|+1)^{3}-(|\mu|+1)\right\}.$$
Here we  used the notations from \cite{126}.
\end{lem}

\medskip

It is known (see, Ponnusamy {\it et al}, \cite{pon}) that
\begin{thm}\label{th4}
Let $f\in\mathcal{C}$, $f^{-1}$  be given by \eqref{eq-4}, and let $\Gamma_{1}$, $\Gamma_{2}$, $\Gamma _{3}$
be its first three logarithmic coefficients. Then
\begin{itemize}
  \item[($i$)] $|\Gamma_{1}|\leq \frac{1}{2}$;
  \item[($ii$)] $|\Gamma_{2}| \leq\frac{1}{4}$;
  \item [($iii$)]$|\Gamma_{3}|\leq \frac16$.
\end{itemize}
All estimates are sharp.
\end{thm}

\medskip

We shall find the lower and the upper bound of the differences $|\Gamma_{2}|-|\Gamma_{1}|$ and $|\Gamma_{3}|-|\Gamma_{2}|$ when $f\in\mathcal{C}$.

\medskip

For the second difference we need to express coefficients of $f\in\mathcal{C}$ in terms of appropriate coefficients of a Schwarz function $\omega$. From the definition of convex functions we have
$$1+ \frac{zf''(z)}{f'(z)}=\frac{1+\omega(z)}{1-\omega(z)}, \qquad |\omega(z)|<1,\, z\in \D, $$
and from here
$$\left[zf'(z)\right]'= \left\{1+2\left[\omega(z)+\omega^{2}(z)+\cdots\right]\right\}\cdot f'(z).$$
If we put $f(z)=z+a_2z^2+a_3z^3+\cdots$ and $\omega(z)=c_{1}z+c_2z^2+c_3z^3+\cdots $, then from the last relation,
after comparing the coefficients, we get
\begin{equation}\label{eq-14}
\begin{split}
a_{2}&=c_{1},\\
a_{3}&=\frac{1}{3}(c_{2}+3c_{1}^{2}), \\
a_{4}&=\frac{1}{6}(c_{3}+5c_{1}c_{2}+6c_{1}^{3}).
\end{split}
\end{equation}
Now, we are ready to prove the final theorem.

\medskip

\begin{thm}\label{th4}
Let $f\in\mathcal{C}$ and let  $f^{-1}$  be given by \eqref{eq-4}. Also, let $\Gamma_{1}$, $\Gamma_{2}$, $\Gamma _{3}$
be its initial logarithmic coefficients. Then
\begin{itemize}
  \item[($i$)] $-\frac{\sqrt{10}}{10}\leq |\Gamma_{2}|-|\Gamma_{1}|\leq \frac{1}{6}.$
\item[($ii$)]  $|\Gamma_{3}|-|\Gamma_{2}|\leq\frac{2\sqrt{10}}{75}.$
\end{itemize}
All inequalities are sharp.
\end{thm}

\medskip

\begin{proof} $ $
\begin{itemize}
  \item[(i)]  Using similar consideration as in Theorem \ref{th3}, the result of Lemma \ref{lem-2}, and $|a_{2}|\leq1$, for the upper bound we get
\[
\begin{split}
|\Gamma_{2}|-|\Gamma_{1}|&=\frac{1}{2}\left|-a_{3}+\frac{3}{2}a_{2}^{2}\right|-\frac{1}{2}|-a_{2}|\\
&=\frac{1}{2}\left|(-a_{3}+a_{2}^{2})+\frac{1}{2}a_{2}^{2}\right|-\frac{1}{2}|a_{2}|\\
&\leq\frac{1}{2}|a_{3}-a_{2}^{2}|+ \frac{1}{4}|a_{2}|^{2}-\frac{1}{2}|a_{2}|\\
&\leq \frac{1}{6}(1-|a_{2}|^{2})+\frac{1}{4}|a_{2}|^{2}-\frac{1}{2}|a_{2}|\\
&\leq\frac{1}{6}.
\end{split}
\]
The inequality is sharp for the function $f_{0}(z)=z+\frac{1}{3}z^{3}+\cdots$ given in Lemma \ref{lem-1}.

\medskip

The lower bound of the inequality is equivalent with
\begin{equation}\label{eq-15}
\left|-a_{3}+\frac{3}{2}a_{2}^{2}\right|\geq|a_{2}|-\sqrt{\frac{2}{5}}.
\end{equation}
If $0\leq |a_{2}|<\sqrt{2/5}$, then the inequality \eqref{eq-15} obviously holds.
Now, let $\sqrt{2/5}\leq |a_{2}|\leq1$. Then, using Lemma \ref{lem-2}, since
\[
\begin{split}
\left|-a_{3}+\frac{3}{2}a_{2}^{2}\right| &\geq \frac{1}{2}|a_{2}|^{2}-|-a_{3}+a_{2}^{2}| \\
&\geq \frac{1}{2}|a_{2}|^{2}-\frac{1}{3}(1-|a_{2}|^{2})\\
&=\frac{5}{6}|a_{2}|^{2}-\frac{1}{3},
\end{split}
\]
we conclude that to prove \eqref{eq-15} it is enough to show that
$$ \frac{5}{6}|a_{2}|^{2}-\frac{1}{3}\geq |a_{2}|-\sqrt{\frac{2}{5}}.$$
The last inequality is equivalent to
$$ \left(|a_{2}|-\sqrt{\frac{2}{5}}\right)\left(\frac{5}{6}\left(|a_{2}|+\sqrt{\frac{2}{5}}\right)-1\right)\geq0,$$
which is true since by assumption $|a_{2}| -\sqrt{\frac{2}{5}}\geq 0$  and
$\frac{5}{6}\left(|a_{2}|+\sqrt{\frac{2}{5}}\right)-1\geq \frac{5}{6}\cdot2\sqrt{\frac{2}{5}}-1=\sqrt{\frac{10}{9}}-1>0$. This proves the inequality \eqref{eq-15} and the lower bound of (i).

\medskip

  \item[(ii)]  In the proof of Theorem \ref{th3} we obtained that $\Gamma_{2}=\frac{1}{2}\left(-a_{3}+\frac{3}{2}a_{2}^{2}\right)$ and
$\Gamma_{3}=\frac{1}{2}\left(-a_{4}+4a_{2}a_{3}-\frac{10}{3}a_{2}^{3}\right)$.
Since $|a_{2}|\leq1$, so also $|a_{2}|\leq\frac54$ and $-1\leq -\frac45 |a_2|$. Therefore,
\[
\begin{split}
|\Gamma_{3}|-|\Gamma_{2}|&\leq |\Gamma_{3}|-\frac45 |a_{2}||\Gamma_{2}|\\
&\leq \left|\Gamma_{3}+\frac45 a_{2}\Gamma_{2}\right|\\
&=\frac{1}{2}\left|a_{4}-\frac{16}{5}a_{2}a_{3}+\frac{32}{15}a_{2}^{3}\right|.
\end{split}
\]
Using \eqref{eq-14} and Lemma \ref{lem-2} (case $D_{8}$), after some calculations, we get
$$|\Gamma_{3}|-|\Gamma_{2}|\leq \frac{1}{12} \left|c_{3}- \frac75c_{1}c_{2}-\frac25c_{1}^{3}\right|\leq\frac{1}{12}\cdot\frac{8}{5}\sqrt{\frac{2}{5}}=\frac{2\sqrt{10}}{75}.$$
\end{itemize}

\medskip

For showing the sharpness, it is enough to consider the function $f_\lambda$ from Lemma  \ref{lem-1} with $\lambda=\sqrt{2/5}$, i.e.,
$$f_{\sqrt{2/5}}(z)=z+\sqrt{\frac{2}{5}}z^{2}+\frac{3}{5}z^{3}+\frac45\sqrt{\frac{2}{5}}z^{4}\cdots,$$
such that
\[a_2=\sqrt{\frac{2}{5}}\ ,\quad a_3=\frac{3}{5}\ ,\quad\mbox{and} \quad  a_4=\frac45\sqrt{\frac{2}{5}}\ ,\]
and further,
\[\Gamma_1=-\frac{\sqrt{10}}{10}\ ,\quad \Gamma_2=0\, \quad\mbox{and} \quad  \Gamma_3=\frac{2\sqrt{10}}{75}\ .\]
The above means that all inequalities in Theorem \ref{th4} are sharp.
\end{proof}

\end{document}